\documentclass{amsart}
\usepackage[latin1]{inputenc}
\usepackage{amssymb}
\usepackage{amsfonts}
\usepackage{enumerate}
\usepackage{enumitem}
\usepackage{indentfirst}
\usepackage{amsmath}
\usepackage{amsthm}
\usepackage[all]{xy}
\usepackage{mathrsfs}

\newcommand{\pups}[3][-1pt]{{}^{#2}\hspace{#1}#3}

\newcommand{\FF}{\mathbb{F}}
\newcommand{\GG}{\mathbb{G}}
\newcommand{\NN}{\mathbb{N}}

\newcommand{\CC}{\mathbb{C}}
\newcommand{\PP}{\mathbb{P}}
\newcommand{\Ucal}{\mathcal{U}}
\newcommand{\Ocal}{\mathcal{O}}
\newcommand{\restr}[2]{\left.{#1}\right|_{#2}}

\DeclareMathOperator{\Antisym}{Antisym}
\DeclareMathOperator{\Aut}{Aut}
\DeclareMathOperator{\PGL}{\PP GL}

\DeclareMathOperator{\Sp}{Sp}
\DeclareMathOperator{\Mat}{Mat}

\DeclareMathOperator{\SL}{SL}

\DeclareMathOperator{\diag}{diag}

\DeclareMathOperator{\gen}{span}
\DeclareMathOperator{\rank}{rank}
\DeclareMathOperator{\ch}{ch}

\newtheorem{thm}{Theorem}[section]
\newtheorem{prop}{Proposition}[section]
\newtheorem{lemma}{Lemma}[section]
\newtheorem{cor}{Corollary}[section]

\begin{document}

\title[On Varieties of Lines on Linear Sections of Grassmannians]{On Varieties of Lines on Linear Sections of Grassmannians}
\author{Rafael Lucas de Arruda}
\date{\today}

\maketitle

\begin{abstract}
General linear sections of codimension $2$ of the Grassmannians $\GG(1, 4)$ and $\GG(1, 5)$ appear in the classification of Fano manifolds of high index. Unlike Grassmannians, these manifolds are not homogeneous. Nevertheless, their automorphisms groups have finitely many orbits. In this work we first compute the orbits of these actions. Then we give a description of the variety of lines (under the Plücker embedding) passing through a fixed point in each orbit of the action. As an application we show that these Fano manifolds are not weakly $2$-Fano, completing the classification of weakly $2$-Fano manifolds of high index, initiated by Araujo and Castravet in \cite{araujo-castravet}.
\end{abstract}

\tableofcontents

\section{Introduction}
\label{sec:intro}

A \textit{Fano} manifold is a smooth, complex, projective variety $X$ with ample anticanonical class, $-K_{X} > 0$. The Fano condition has strong implications on the geometry of $X$. For example, any Fano manifold $X$ is \textit{rationally connected}, that is, any two general points of $X$ can be connected by a rational curve contained in $X$. Examples of Fano manifolds are hypersurfaces of degree $d \leq n$ in $\PP^{n}$, Grassmannians $\GG(k, N)$ of $k$-planes of $\PP_{\CC}^{N}$, and their general linear sections $\GG(k, N) \cap H^{l}$ of codimension $l$, for $1 \leq k \leq (N - 1)/2$ and $l \leq N$.

The \textit{index} $i_{X}$ of a Fano manifold $X$ of dimension $n$ is defined by
\[ i_{X} = \max \{ m \in \NN \mid -K_{X} = mH, \ \textrm{for some Cartier divisor $H$} \}. \]
A classical result of Kobayashi and Ochiai asserts that $i_{X} \leq n + 1$. Fano manifolds with index $i_{X} \geq n - 2$ are completely classified (see \cite{ag-iv-fanovarieties}). Among those are the general linear sections $\GG(1, 4) \cap H^{2}$ and $\GG(1, 5) \cap H^{2}$.

Recall that, under the Plücker embedding, the variety of lines of $\GG(k, N)$ through a fixed point is isomorphic to $\PP^{k} \times \PP^{N-k-1}$. It does not depend on the choice of the point because $\GG(k, N)$ is homogeneous. Although the linear sections $\GG(1, 4) \cap H^{2}$ and $\GG(1, 5) \cap H^{2}$ are not homogeneous, their automorphisms groups have finitely many orbits. Their varieties of lines (under the Plücker embedding) through a fixed point is an intersection of two divisors of type $(1, 1)$ in $\PP^{1} \times \PP^{2}$ and $\PP^{1} \times \PP^{3}$, respectively. Our aim is to describe these varieties. \\

Let $X = \GG(1, 4) \cap H^{2}$ be the intersection of $\GG(1, 4)$ with a general linear subspace $H^{2}$ of codimension $2$, under the Plücker embedding. The four orbits of the automorphism group $\Aut(X) \subset \PGL(5, \CC)$ can be described as follows (see Section \ref{sec:preli} for details). There is a smooth conic $C \subset \PP^{4}$ such that:

\begin{itemize}
  \item the orbit $o_{1}$ consists of the lines in $\PP^{4}$ tangent to $C$;
  \item the orbit $o_{2}$ consists of the lines in $\PP^{4}$ that lie in the plane $P$ generated by $C$ but are not tangent to $C$;
  \item the orbit $o_{3}$ consists of the lines in $\PP^{4}$ that intersect $C$ but do not lie in $P$;
  \item the orbit $o_{4}$ consists of the lines in $\PP^{4}$ that do not intersect $P$ (this is a dense orbit).
\end{itemize}
We have the following description of the variety of lines passing through a fixed point $x \in X$:

\begin{thm}
\label{thm:zx_g25_h2}
Let $X = \GG(1, 4) \cap H^{2}$ be a general linear section of codimension $2$ of $\GG(1, 4)$, and let $Z_{x} \subset \PP^{1} \times \PP^{2}$ be the variety of lines on $X$ passing through a point $x \in X$. Then $Z_{x}$ has pure dimension $1$, and its numerical class in $N_{1}(\PP^{1} \times \PP^{2})$ is $[Z_{x}] \equiv 2[L_{1}] + [L_{2}]$, where $[L_{1}]$ is the class of a line in a fiber of the first projection $\PP^{1} \times \PP^{2} \rightarrow \PP^{1}$ and $[L_{2}]$ the class of a fiber of the second projection. For $x$ in each orbit of the action of $\Aut(X)$ on $X$ we have the following description of $Z_{x}$:

\begin{itemize}
  \item for $x \in o_{1}$, $Z_{x}$ has two irreducible components, all rational. One of them has numerical class $[L_{1}]$ (and multiplicity $2$), and the other one has numerical class $[L_{2}]$;
  \item for $x \in o_{2}$, $Z_{x}$ has three irreducible components, all rational. Two of them has numerical class $[L_{1}]$, and the other one has numerical class $[L_{2}]$;
  \item for $x \in o_{3}$, $Z_{x}$ has two irreducible components, all rational. One of them has numerical class $[L_{1}]$ and the other one has numerical class $[L_{1}] + [L_{2}]$;
  \item for $x \in o_{4}$, $Z_{x} \cong \PP^{1}$.
\end{itemize}
\end{thm}

Now let $Y = \GG(1, 5) \cap H^{2}$ be the intersection of $\GG(1, 5)$ with a general linear space $L = H^{2}$ of codimension $2$, under the Plücker embedding. The four orbits of the automorphism group $\Aut(Y) \subset \PGL(6, \CC)$ can be described as follows (see Section \ref{sec:preli} for details). There are three disjoint lines $l_{1}, l_{2}, l_{3}$ of $\PP^{5}$, $3$-planes $V_{j}$, for $j = 1, 2, 3$, generated by the ${l_{i}}'s$ with $i \neq j$, and $V = \cup_{i=1}^{3} V_{i}$ such that:

\begin{itemize}
  \item the orbit $o_{1}$ consists of the lines in $\PP^{5}$ that intersect two ${l_{i}}'s$;
  \item the orbit $o_{2}$ consists of the lines in $\PP^{5}$ that intersect only one $l_{i}$;
  \item the orbit $o_{3}$ consists of the lines in $\PP^{5}$ that intersect $V \setminus \cup_{i=1}^{3} l_{i}$;
  \item the orbit $o_{4}$ consists of the lines in $\PP^{5}$ that do not intersect $V$ (this is a dense orbit).
\end{itemize}
We have the following description of the variety of lines passing through a fixed point $x \in Y$:

\begin{thm}
\label{thm:zx_g26_h2}
Let $Y = \GG(1, 5) \cap H^{2}$ be a general linear section of codimension $2$ of $\GG(1, 5)$, and let $Z_{x} \subset \PP^{1} \times \PP^{3}$ be the variety of lines on $Y$ passing through a point $x \in Y$. Then $Z_{x}$ has pure dimension $2$, and its numerical class in $N_{2}(\PP^{1} \times \PP^{3})$ is $[Z_{x}] \equiv 2[P] + [L]$, where $[P]$ is the class of a plane in a fiber of the first projection $\PP^{1} \times \PP^{3} \rightarrow \PP^{1}$ and $[L]$ the class of the inverse image under the second projection of a line in $\PP^{3}$. For $x$ in each orbit of the action of $\Aut(Y)$ on $Y$ we have the following description of $Z_{x}$:

\begin{itemize}
  \item for $x \in o_{1}$, $Z_{x}$ has three irreducible components. Two of them have numerical class $[P]$, and the other one has numerical class $[L]$;
  \item for $x \in o_{2}$, $Z_{x}$ has two irreducible components. One of them has numerical class $[P]$, and the other one has numerical class $[P] + [L]$;
  \item for $x \in o_{3}$, $Z_{x}$ is isomorphic to the blowup of a quadric cone in $\PP^{3}$ at the \linebreak vertex, or equivalently, isomorphic to the Hirzebruch surface \linebreak $\FF_{2} = \PP(\Ocal_{\PP^{1}} \oplus \Ocal_{\PP^{1}}(-2))$;
  \item for $x \in o_{4}$, $Z_{x}$ is isomorphic to a smooth quadric in $\PP^{3}$.
\end{itemize}
\end{thm}

As an application of these results, in Section \ref{sec:w2fano} we complete the classification of weakly $2$-Fano manifolds, initiated in \cite{araujo-castravet} by Araujo and Castravet. We prove that the general linear sections $\GG(1, 4) \cap H^{2}$ and $\GG(1, 5) \cap H^{2}$ are not weakly $2$-Fano. \\

\noindent \textbf{Notation.} For a $(k + 1)$-dimensional vector subspace $W$ of $\CC^{N+1}$ we will denote by $\PP(W)$ its corresponding $k$-plane in $\PP_{\CC}^{N}$, and by $[\PP(W)]$ its corresponding point in the Grassmannian $\GG(k, N)$. We will denote by $\PP(\mathscr{E})$ the projectivization of a vector bundle $\mathscr{E}$, the variety of lines of $\mathscr{E}$. \\

\noindent \textbf{Acknowledgment.} I thank Carolina Araujo for the introduction to the subject and many helpful comments.

\section{Preliminaries}
\label{sec:preli}

We will consider the Grassmannian $\GG(k, N)$ of $k$-planes of $\PP_{\CC}^{N}$ embedded into $\PP(\bigwedge^{k+1}\CC^{N+1})$ by the Plücker embedding
$$
\begin{array}{rcl}
\GG(k, N) & \longrightarrow & \PP(\bigwedge^{k+1}\CC^{N+1}) \\
\PP(\gen\{ u_{1}, \ldots, u_{k+1} \}) & \longmapsto & \PP(u_{1} \wedge \cdots \wedge u_{k+1}).
\end{array}
$$
Denote by $H^{l}$ a general linear subspace of codimension $l$ of $\PP(\bigwedge^{k+1}\CC^{N+1})$. \\

\noindent \textbf{The Automorphism Group of $\GG(1, N) \cap H^{l}$.} Given a subvariety $Y$ of a variety $X$, we will denote by $\Aut(Y, X)$ the group of automorphisms of $X$ that induce automorphisms of $Y$, that is,
$$\Aut(Y, X) = \{ \varphi \in \Aut(X) \mid \varphi(Y) \subset Y \}.$$

It is well known that $\Aut(\GG(k, N)) \cong \Aut(\GG(k, N), \PP(\bigwedge^{k+1}\CC^{N+1})) \cong \linebreak \PGL(N + 1, \CC)$, for $k < (N + 1)/2$ (see, for example, \cite[Thm. 10.19]{harris}). For general linear sections of Grassmannians, J. Piontkowski and A. Van de Ven in \cite{piont-van-de-ven} determined the automorphism groups of $\GG(1, N) \cap H^{2}$. Their first result is the following:

\begin{thm}[{\cite[Thm. 1.2 and Cor. 1.3]{piont-van-de-ven}}]
\label{thm:autlinsecgrass}

For $N \geq 4$ and a general linear subspace $H^{l} \subset \PP(\bigwedge^{2}\CC^{N+1})$ of codimension $l \leq 2N - 5$,
\[ \Aut(\GG(1, N) \cap H^{l}) \cong \textstyle \Aut(\GG(1, N) \cap H^{l}, H^{l}). \]
If $l \leq N - 2$, then
\[ \Aut(\GG(1, N) \cap H^{l}) \cong \textstyle \Aut(\GG(1, N) \cap H^{l}, \PP(\bigwedge^{2} \CC^{N+1})) \cap \Aut(H^{l}, \PP(\bigwedge^{2} \CC^{N+1})). \]
\end{thm}

The theorem says that the automorphisms of $\GG(1, N) \cap H^{l}$ are the automorphisms in $\PGL(N + 1 , \CC)$ such that their induced action on the dual space $\PP(\bigwedge^{2}\CC^{N+1})^{\ast}$ fixes $(H^{l})^{\ast}$. If $\{ e_{1}, \ldots, e_{N+1} \}$ is a basis of $\CC^{N+1}$ and $E_{ij} \in \Mat(N + 1, \CC)$ the matrix with the entry $(i, j)$ equal to $1$ and otherwise equal to $0$, then
$$
\begin{array}{rcl}
\left( \bigwedge^{2}\CC^{N+1} \right)^{\ast} & \longrightarrow & \Antisym(N + 1, \CC) \\
\sum_{i,j} \lambda_{i,j}(e_{i} \wedge e_{j})^{\ast} & \longmapsto & \frac{1}{2} \sum_{i,j} \lambda_{i,j}(E_{i,j} - E_{j,i})
\end{array}
$$
is an isomorphism of vector spaces; that allows us to identify $\PP(\bigwedge^{2}\CC^{N+1})^{\ast}$ with $\PP(\Antisym(N + 1, \CC))$. A point $[\PP(\gen\{ p, q \})] \in \GG(1, N)$ is contained in a hyperplane $H = \PP(A)^{\ast}$, with $A \in \Antisym(N + 1, \CC)$, if and only if, $pA\pups{t}{q} = 0$. The action of an automorphism $\PP(T) \in \PGL(N + 1, \CC)$ on $\PP(\Antisym(N + 1, \CC))$ is given by
$$
\begin{array}{rcl}
\PP(\Antisym(N + 1, \CC)) & \longrightarrow & \PP(\Antisym(N + 1, \CC)) \\
\PP(A) & \longmapsto & \PP(\pups[0pt]{t}{T}^{-1}AT^{-1}).
\end{array}
$$
Therefore, a linear subspace $H^{l} \subset \PP(\bigwedge^{2}\CC^{N+1})$ of codimension $l$, dually given by $\PP(\gen\{ A_{1}, \ldots, A_{l} \})^{\ast}$, is preserved under $T$ if and only if,
\begin{equation}
\pups[0pt]{t}{T}^{-1}A_{i}T^{-1} \in \gen\{ A_{1}, \ldots, A_{l} \}, \quad \textrm{for all} \ i = 1, \ldots, l. \label{eq:autlin}
\end{equation}

The second step in the task of determining the automorphism groups is done separately for the different cases using the above description. Consider the case $l = 1$. The automorphism group of $Y = \GG(1, 2n - 1) \cap H$, where $H = \PP(A)^{\ast}$ is a general hyperplane of $\PP(\bigwedge^{2}\CC^{2n})$, is isomorphic to the group $\Sp(2n, \CC) / \{ \pm I\}$, where $\Sp(2n, \CC)$ denotes the symplectic group associated to $A$. Its action on $Y$ is homogeneous (see \cite[Prop. 2.1]{piont-van-de-ven}). The automorphism group of $X = \GG(1, 2n) \cap H$, where $H = \PP(A)^{\ast}$ is a general hyperplane of $\PP(\bigwedge^{2}\CC^{2n+1})$, act on $X$ with two orbits (see \cite[Prop. 5.3]{piont-van-de-ven} for details). \\

The next results concern automorphism groups of linear sections of codimension $l = 2$. \\

\noindent \textbf{The Case $\GG(1, 2n - 1) \cap H^{2}$.}

\begin{thm}[{\cite[Thm. 3.5]{piont-van-de-ven}}]
\label{thm:autg22n}

For $n \geq 3$ the automorphism group of the intersection of $\GG(1, 2n - 1)$ with a general linear subspace of codimension $2$ of $\PP(\bigwedge^{2}\CC^{2n})$ is isomorphic to the subgroup of $\PGL(2n, \CC)$ that consists of the elements
$$P_{\sigma} \cdot
\left(
  \begin{array}{ccc}
    t_{1} & & \raisebox{-5pt}{\textrm{\huge $0$}} \\[-2pt]
     & \ddots &  \\[-2pt]
    \textrm{\huge $0$} & & t_{n} \\
  \end{array}
\right)
$$
where $t_{1}, \ldots, t_{n} \in \SL(2, \CC)$ and $P_{\sigma}$ is the identity for $n \geq 5$ and otherwise defined by $P_{\sigma}(e_{2i}) = e_{2\sigma(i)}, P_{\sigma}(e_{2i-1}) = e_{2\sigma(i)-1}$,
$$
\textrm{for} \ \sigma \in
\left\{
  \begin{array}{ll}
    S(n), & \hbox{if $n = 3$;} \\
    \{ (1\ 2\ 3\ 4), (2\ 1\ 4\ 3), (3\ 4\ 1\ 2), (4\ 3\ 2\ 1) \}, & \hbox{if $n = 4$.}
  \end{array}
\right.
$$
\end{thm}

In order to describe the orbits of the action of $\Aut(\GG(1, 5) \cap H^{2})$ on \linebreak $\GG(1, 5) \cap H^{2}$ let us explain the generality condition on $H^{2}$ assumed in the \linebreak theorem. A linear subspace $L = H^{2} = H_{1} \cap H_{2} \subset \PP(\bigwedge^{2}\CC^{2n})$ of codimension $2$, given by the intersection of two distinct hyperplanes $H_{1} = \PP(A)^{\ast}, H_{2} = \PP(B)^{\ast}$, with $A, B \in \Antisym(2n, \CC)$, is dual to the line $L^{\ast} = \PP(\lambda A - \mu B) \subset \PP(\Antisym(2n, \CC))$. The dual Grassmannian of $\GG(1, 2n - 1)$ is given by
$$\GG(1, 2n - 1)^{\ast} = \{ \PP(C) \in \PP(\Antisym(2n, \CC)) \mid \rank(C) \leq 2n - 2 \}$$
and it is an irreducible hypersurface of degree $n$. Therefore $L^{\ast}$ intersects \linebreak $\GG(1, 2n - 1)^{\ast}$ in at most $n$ points. We will say that $L = H^{2}$ is \textit{general} if $L^{\ast}$ and $\GG(1, 2n - 1)^{\ast}$ have $n$ distinct points in common, which we denote by \linebreak $\PP(\lambda_{i}A - \mu_{i}B), i = 1, \ldots, n$. The corresponding hyperplanes $H_{i} = \PP(\lambda_{i}A - \mu_{i}B)^{\ast}$, are tangent to the Grassmannian $\GG(1, 2n - 1)$ at the points $[l_{i}]$, where \linebreak $l_{i} = \PP(\ker(\lambda_{i}A - \mu_{i}B)) \subset \PP^{2n-1}$ are called \textit{exceptional lines}.

Returning to the particular case $\GG(1, 5) \cap H^{2}$, for $j = 1, 2, 3$, let $V_{j}$ be the $3$-plane in $\PP^{5}$ generated by the exceptional lines ${l_{i}}'s$ with $i \neq j$. Let $V = \cup_{i=1}^{3} V_{i}$. Up to changing coordinates we can write
\begin{equation}
\label{eq:basisc6}
A =
\left(
  \begin{array}{cccccc}
    0 & -1 & 0 & 0 & 0 & 0 \\
    1 & 0 & 0 & 0 & 0 & 0 \\
    0 & 0 & 0 & -1 & 0 & 0 \\
    0 & 0 & 1 & 0 & 0 & 0 \\
    0 & 0 & 0 & 0 & 0 & -1 \\
    0 & 0 & 0 & 0 & 1 & 0 \\
  \end{array}
\right) \quad \textrm{and} \quad B =
\left(
  \begin{array}{cccccc}
    0 & -1 & 0 & 0 & 0 & 0 \\
    1 & 0 & 0 & 0 & 0 & 0 \\
    0 & 0 & 0 & 0 & 0 & 0 \\
    0 & 0 & 0 & 0 & 0 & 0 \\
    0 & 0 & 0 & 0 & 0 & 1 \\
    0 & 0 & 0 & 0 & -1 & 0 \\
  \end{array}
\right)
\end{equation}
(see \cite{donagi} or \cite[Prop. 3.2]{piont-van-de-ven}). The exceptional lines are $l_{1} = \PP(\gen\{ e_{1}, e_{2} \})$, $l_{2} = \PP(\gen\{ e_{3}, e_{4} \})$ and $l_{3} = \PP(\gen\{ e_{5}, e_{6} \})$.

\begin{lemma}
\label{lemma:g26trsv}

The automorphism group $\Aut(\GG(1, 5) \cap H^{2}) \subset \PGL(6, \CC)$ acts transitively on $\PP^{5} \setminus V$.
\end{lemma}

\begin{proof}
It is sufficient to prove the existence of an automorphism in $\Aut(\GG(1, 5) \cap H^{2})$ mapping $p = (1: 0: 1: 0: 1: 0)$ to a given point $q = (q_{1}: q_{2}: q_{3}: q_{4}: q_{5}: q_{6}) \in \PP^{5} \setminus V$. Any block diagonal matrix $T = \diag(t_{1}, t_{2}, t_{3}) \in \Mat(6, \CC)$ with
$$
t_{1} =
\left(
  \begin{array}{cc}
    q_{1} & a_{12} \\
    q_{2} & a_{22} \\
  \end{array}
\right),
t_{2} =
\left(
  \begin{array}{cc}
    q_{3} & a_{34} \\
    q_{4} & a_{44} \\
  \end{array}
\right),
t_{3} =
\left(
  \begin{array}{cc}
    q_{5} & a_{56} \\
    q_{6} & a_{66} \\
  \end{array}
\right)
\in \Mat(2, \CC)
$$
defines a projective transformation $\PP(T)$ of $\PP^{5}$ such that $\PP(T)(p) = q$. Since $q \not\in V$, that is, $(q_{1}, q_{2}), (q_{3}, q_{4}), (q_{5}, q_{6}) \neq (0, 0)$, we can choose the ${a_{ij}}'s$ satisfying $\det(t_{1}) = \det(t_{2}) = \det(t_{3}) = 1$. Then we have $\PP(T) \in \Aut(\GG(1, 5) \cap H^{2})$.
\end{proof}

\begin{lemma}
\label{lemma:g26trsv1}

The automorphism group $\Aut(\GG(1, 5) \cap H^{2}) \subset \PGL(6, \CC)$ acts transitively on $V \setminus \cup_{i=1}^{3} l_{i}$.
\end{lemma}

\begin{proof}
Let $p, q \in V \setminus \cup_{i=1}^{3} l_{i}$. By a permutation $P_{\sigma}$ on the exceptional lines we can suppose $p = (0: 0: p_{3}: p_{4}: p_{5}: p_{6}), q = (0: 0: q_{3}: q_{4}: q_{5}: q_{6}) \in V_{1} \setminus \cup_{i=1}^{3} l_{i}$. Let $T = \diag(t_{1}, t_{2}, t_{3}) \in \Mat(6, \CC)$ be a block diagonal matrix with
$$
t_{1} =
\left(
  \begin{array}{cc}
    a_{11} & a_{12} \\
    a_{21} & a_{22} \\
  \end{array}
\right),
t_{2} =
\left(
  \begin{array}{cc}
    a_{33} & a_{34} \\
    a_{43} & a_{44} \\
  \end{array}
\right),
t_{3} =
\left(
  \begin{array}{cc}
    a_{55} & a_{56} \\
    a_{65} & a_{66} \\
  \end{array}
\right)
\in \Mat(2, \CC).
$$
The induced projective transformation $\PP(T)$ satisfies $\PP(T)(p) = q$ if
$$(p_{3}a_{33} + p_{4}a_{34})e_{3} + (p_{3}a_{43} + p_{4}a_{44})e_{4} + (p_{5}a_{55} + p_{6}a_{56})e_{5} + (p_{5}a_{65} + p_{6}a_{66})e_{6} = \textstyle \sum_{i=3}^{6} q_{i}e_{i}.$$
Since $p, q \not\in \cup_{i=1}^{3} l_{i}$ we can choose the ${a_{ij}}'s$ satisfying the condition above and the three additional ones $\det(t_{1}) = \det(t_{2}) = \det(t_{3}) = 1$.
\end{proof}

\begin{lemma}
\label{lemma:g26intv}

Let $x = [l] \in \GG(1, 5) \cap H^{2}$. If the line $l$ intersects one $V_{i}$, then it intersects the other two.
\end{lemma}

\begin{proof}
By a permutation $P_{\sigma}$ on the exceptional lines we can suppose that $l \cap V_{1} \neq \emptyset$. Then $l = \PP(\gen\{ p, q \})$, where $p = (0, 0, p_{3}, p_{4}, p_{5}, p_{6})$ and $q = (q_{1}, q_{2}, q_{3}, q_{4}, q_{5}, q_{6})$, with $p_{5}q_{6} - p_{6}q_{5} = 0$ and $p_{3}q_{4} - p_{4}q_{3} = 0$. Denote by $M_{ij}$, $1 \leq i < j \leq 6$, the $(i, j)$-minor of the matrix
$$
\left(
  \begin{array}{cccccc}
    0 & 0 & p_{3} & p_{4} & p_{5} & p_{6} \\
    q_{1} & q_{2} & q_{3} & q_{4} & q_{5} & q_{6} \\
  \end{array}
\right).
$$
Note that $M_{34} = M_{56} = 0$. We have
\begin{equation*}
\begin{split}
q_{3}p - p_{3}q & = (-p_{3}q_{1}, -p_{3}q_{2}, q_{3}p_{4} - p_{4}q_{3}, q_{3}p_{4} - p_{3}q_{4}, q_{3}p_{6} - p_{3}q_{5}, q_{3}p_{6} - p_{3}q_{6}) \\
                & = (M_{13}, M_{23}, 0, 0, -M_{35}, -M_{36}) \\
q_{4}p - p_{4}q & = (-p_{4}q_{1}, -p_{4}q_{2}, q_{4}p_{3} - p_{4}q_{3}, q_{4}p_{4} - p_{4}q_{4}, q_{4}p_{5} - p_{4}q_{5}, q_{4}p_{6} - p_{4}q_{6}) \\
                & = (M_{14}, M_{24}, 0, 0, -M_{45}, -M_{46}) \\
q_{5}p - p_{5}q & = (-p_{5}q_{1}, -p_{5}q_{2}, q_{5}p_{3} - p_{5}q_{3}, q_{5}p_{4} - p_{5}q_{4}, q_{5}p_{5} - p_{5}q_{5}, q_{5}p_{6} - p_{5}q_{6}) \\
                & = (M_{15}, M_{25}, M_{36}, M_{46}, 0, 0) \\
q_{6}p - p_{6}q & = (-p_{6}q_{1}, -p_{6}q_{2}, q_{6}p_{3} - p_{6}q_{3}, q_{6}p_{4} - p_{6}q_{4}, q_{6}p_{5} - p_{6}q_{5}, q_{6}p_{6} - p_{6}q_{6}) \\
                & = (M_{16}, M_{26}, M_{35}, M_{45}, 0, 0).
\end{split}
\end{equation*}
Since $p$ and $q$ are linearly independents, it is clear that $l$ intersects $V_{2}$ and $V_{3}$.
\end{proof}

\begin{prop}
\label{prop:orbg26}

The action of $\Aut(\GG(1, 5) \cap H^{2})$ on $X = \GG(1, 5) \cap H^{2}$ has four orbits:

\begin{itemize}
  \item $o_{1} = \{ x = [l] \in X \mid \textrm{$l$ intersects two exceptional lines} \}$;
  \item $o_{2} = \{ x = [l] \in X \mid \textrm{$l$ intersects only one exceptional line} \}$;
  \item $o_{3} = \{ x = [l] \in X \mid \textrm{$l$ intersects $V \setminus \cup_{i=1}^{3} l_{i}$} \}$;
  \item $o_{4} = \{ x = [l] \in X \mid \textrm{$l$ does not intersect $V$} \}$.
\end{itemize}
\end{prop}

Note that none of the exceptional lines lies in $X$ and that no line in $\PP^{5}$ intersects the three exceptional lines.

\begin{proof}
Since any automorphism permutes the exceptional lines, it is clear by the geometric description that the four kinds of lines described lie in different orbits. \\

Let $x = [l], x'=[l'] \in o_{1}$. By a permutation $P_{\sigma}$ on the exceptional lines we reduce to particular cases. The first one when there is only one exceptional line intersecting both $l$ and $l'$, say $l = \PP(\gen\{ re_{1} + se_{2}, te_{3} + ue_{4} \})$ and $l' = \PP(\gen\{ r'e_{1} + s'e_{2}, t'e_{5} + u'e_{6} \})$. Let $T = \diag(t_{1}, t_{2}, t_{3}) \in \Mat(6, \CC)$ be a block diagonal matrix with
$$
t_{1} =
\left(
  \begin{array}{cc}
    a_{11} & a_{12} \\
    a_{21} & a_{22} \\
  \end{array}
\right),
t_{2} =
\left(
  \begin{array}{cc}
    a_{35} & a_{36} \\
    a_{45} & a_{46} \\
  \end{array}
\right),
t_{3} =
\left(
  \begin{array}{cc}
    1 & 0 \\
    0 & 1 \\
  \end{array}
\right)
\in \Mat(2, \CC).
$$
The induced projective transformation $\PP(P_{(2\ 3)} \cdot T)$ satisfies $\PP(P_{(2\ 3)} \cdot T)(re_{1} + se_{2}) = r'e_{1} + s'e_{2}$ and $\PP(P_{(2\ 3)} \cdot T)(te_{3} + ue_{4}) = t'e_{5} + u'e_{6}$ if
\begin{equation*}
\begin{split}
(ra_{11} + sa_{12})e_{1} + (ra_{21} + sa_{22})e_{2} & = r'e_{1} + s'e_{2} \\
(ta_{53} + ua_{54})e_{5} + (ta_{63} + ua_{64})e_{6} & = t'e_{5} + u'e_{6}.
\end{split}
\end{equation*}
Since $(r, s), (t, u), (r', s'), (t', u') \neq (0, 0)$ we can choose the ${a_{ij}}'s$ satisfying the conditions above and the two additional ones $\det(t_{1}) = \det(t_{2}) = 1$. With these values of ${a_{ij}}'s$, $\PP(P_{(2\ 3)} \cdot T)(x) = x'$. The second case, when two exceptional lines intersect both $l$ and $l'$, can be treated similarly. \\

Similarly to the previous case, it is shown that $o_{2}$ is an orbit. \\

Let $x = [l], x'=[l'] \in o_{3}$. Since $l$ and $l'$ do not intersect any of the \linebreak exceptional lines $l_{i}$, by Lemma \ref{lemma:g26intv} and its proof we can suppose \linebreak $l = \PP(\gen\{ p = (0, 0, p_{3}, p_{4}, p_{5}, p_{6}), q = (q_{1}, q_{2}, 0, 0, q_{5}, q_{6}) \})$ with $p_{5}q_{6} - p_{6}q_{5} = 0$ and $l' = \PP( \gen\{ p' = (0, 0, p_{3}', p_{4}', p_{5}', p_{6}'), q' = (q_{1}', q_{2}', 0, 0, q_{5}', q_{6}') \})$ with $p_{5}'q_{6}' - p_{6}'q_{5}' = 0$. By Lemma \ref{lemma:g26trsv1} the group $\Aut(X)$ acts transitively on $\PP^{5} \setminus V$, so we can suppose $p = p' = (0, 0, 1, 0, 1, 0)$, $q_{6} = q_{6}' = 0$ and $q_{5} = q_{5}' = 1$. Hence $l = \PP(\gen\{ p = (0, 0, 1, 0, 1, 0), q = (q_{1}, q_{2}, 0, 0, 1, 0) \})$ and $l' = \PP(\gen\{ p' = (0, 0, 1, 0, 1, 0), q' = (q_{1}', q_{2}', 0, 0, 1, 0) \})$ with $(q_{1}: q_{2}), (q_{1}': q_{2}') \in \PP^{1}$. Let $T = \diag(t_{1}, t_{2}, t_{3}) \in \Mat(6, \CC)$ be a block diagonal matrix with
$$
t_{1} =
\left(
  \begin{array}{cc}
    a_{11} & a_{12} \\
    a_{21} & a_{22} \\
  \end{array}
\right),
t_{2} =
\left(
  \begin{array}{cc}
    a_{33} & a_{34} \\
    a_{43} & a_{44} \\
  \end{array}
\right),
t_{3} =
\left(
  \begin{array}{cc}
    a_{55} & a_{56} \\
    a_{65} & a_{66} \\
  \end{array}
\right)
\in \Mat(2, \CC).
$$
The induced projective transformation $\PP(T)$ satisfies $\PP(T)(p) = p$ and $\PP(T)(q) = q'$ if
\begin{equation*}
\begin{split}
a_{33}e_{3} + a_{43}e_{4} + a_{55}e_{5} + a_{65}e_{6} & = e_{3} + e_{5} \\
(a_{11}q_{1} + a_{12}q_{2})e_{1} + (a_{21}q_{1} + a_{22}q_{2})e_{2} + a_{55}e_{5} + a_{65}e_{6} & = q_{1}'e_{1} + q_{2}'e_{2} + e_{5}.
\end{split}
\end{equation*}
Hence we put $a_{33} = 1, a_{43} = 0, a_{55} = 1, a_{65} = 0, a_{44} = a_{66} = 1$. Since $(q_{1}, q_{2}), \linebreak (q_{1}', q_{2}') \neq (0, 0)$ we can choose the rest of the ${a_{ij}}'s$ satisfying the last equation above and the additional one $\det(t_{1}) = 1$. With these values of ${a_{ij}}'s$, $\PP(T)(x) = x'$. \\

To show that $o_{4}$ is an orbit it is sufficient to find an automorphism in $\Aut(X)$ mapping the line $l_{0} = [\PP(\gen\{ (1, 0, 1, 0, 1, 0), (1, 1, 1, -2, 1, 1) \})]$ to a given \linebreak $x = [l] \subset \PP^{5} \setminus V$. By Lemma \ref{lemma:g26trsv} the group $\Aut(X)$ acts transitively on $\PP^{5} \setminus V$, so we can suppose $l = \PP(\gen\{ (1, 0, 1, 0, 1, 0), (q_{1}, q_{2}, q_{3}, -2q_{2}, q_{5}, q_{2}) \})$. Note that $q_{2} \neq 0$, otherwise $l \cap V_{3} \neq \emptyset$. We can put $q_{2} = 1$. The block diagonal matrix $T = \diag(t_{1}, t_{2}, t_{3}) \in \Mat(6, \CC)$ with
$$
t_{1} =
\left(
  \begin{array}{cc}
    1 & q_{1} - 1 \\
    0 & 1 \\
  \end{array}
\right),
t_{2} =
\left(
  \begin{array}{cc}
    1 & (1 - q_{3})/2 \\
    0 & 1 \\
  \end{array}
\right),
t_{3} =
\left(
  \begin{array}{cc}
    1 & q_{5} - 1 \\
    0 & 1 \\
  \end{array}
\right)
\in \Mat(2, \CC)
$$
defines an automorphism $\PP(T)$ of $X$ such that $\PP(T)(x_{0}) = x$.
\end{proof}

\noindent \textbf{The Case $\GG(1, 2n) \cap H^{2}$.} Now let $L = H^{2} = H_{1} \cap H_{2} \subset \PP(\bigwedge^{2} \CC^{2n+1})$ be a linear subspace of codimension $2$ given by the intersection of two distinct hyperplanes $H_{1} = \PP(A)^{\ast}, H_{2} = \PP(B)^{\ast}$, with $ A, B \in \Antisym(2n + 1, \CC)$. As before, $L$ is dual to the line $L^{\ast} = \PP(\lambda A - \mu B) \subset \PP(\Antisym(2n + 1, \CC))$. The dual Grassmannian of $\GG(1, 2n)$ is given by
$$\GG(1, 2n)^{\ast} = \{ \PP(C) \in \PP(\Antisym(2n + 1, \CC)) \mid \rank(C) \leq 2n - 2 \}$$
and it is a subvariety of codimension $3$. We will say that $L = H^{2}$ is \textit{general} if $L^{\ast}$ does not intersect $\GG(1, 2n)^{\ast}$. Since antisymmetric matrices have even rank, the ones $\lambda A - \mu B$ corresponding to the hyperplanes $H_{(\lambda: \mu)} = \PP(\lambda A - \mu B)^{\ast}$ tangent to $\GG(1, 2n)$ have all corank $1$. The $1$-dimensional kernel $c(\lambda: \mu)$ of $\lambda A - \mu B$ as a point of $\PP^{2n}$ is called \textit{center} of $H_{(\lambda: \mu)}$. The map
$$
\begin{array}{rrcl}
c: & \PP^{1} & \longrightarrow & \PP^{2n} \\
& (\lambda: \mu) & \longmapsto & c(\lambda: \mu) = \PP(\ker(\lambda A - \mu B))
\end{array}
$$
is a parametrization of a rational normal curve $C$ of degree $n$, called \textit{center curve}; the map
$$
\begin{array}{rrcl}
h: & \PP^{1} & \longrightarrow & (\PP^{2n})^{\ast} \\
& (\lambda: \mu) & \longmapsto & h(\lambda: \mu) = \PP\left(\ker\left(
                                                         \begin{array}{c}
                                                           c(\lambda: \mu)A \\
                                                           c(\lambda: \mu)B \\
                                                         \end{array}
                                                       \right)
                                                      \right)
\end{array}
$$
is a parametrization of a rational normal curve $H$ of degree $n - 1$ in the space of hyperplanes containing the center curve.

Any automorphism $\PP(T) \in \Aut(\GG(1, 2n) \cap H^{2}) \subset \PGL(2n + 1, \CC)$ maps the center curve onto itself and also the projective space $P \cong \PP^{n}$ spanned by the center curve onto itself. Hence $\PP(T)$ induces an automorphism on the rational normal curve $H$ in the dual projective space $(\PP^{2n}/P)^{\ast}$. The group of automorphisms of $\PP^{n}$ fixing a rational normal curve of degree $n$ is isomorphic to $\PGL(2, \CC)$ (see, for example, \cite[Example 10.12]{harris}). In other words, we know to describe $\PP(T)$ when restricted to $C$ and $H$. With such a description and using (\ref{eq:autlin}) we obtain:

\begin{thm}[{\cite[Thm. 6.6]{piont-van-de-ven}}]
\label{thm:autg22n1}

The automorphism group of the intersection of $\GG(1, 2n)$ with a general linear subspace of codimension $2$ of $\PP(\bigwedge^{2} \CC^{2n+1})$ is isomorphic to the subgroup of $\PGL(2n + 1, \CC)$ that consists of the elements
$$
\left(
  \begin{array}{cc}
    \alpha I_{n} & 0 \\
    S & I_{n+1} \\
  \end{array}
\right)
\cdot
\left(
  \begin{array}{cc}
    \pups{t}{t_{n}}^{-1} & 0 \\
    0 & t_{n+1} \\
  \end{array}
\right),
$$
where $\alpha \in \CC^{\ast}$, $S \in \Mat((n+1) \times n, \CC)$ with $s_{ij} = s_{kl}$ for $i + j = k + l$, $t_{n} \in \Aut(H, \PP^{n-1})$ and $t_{n+1} \in \Aut(C, \PP^{n})$.
\end{thm}

Now we specialize to the case $n = 2$. Up to changing coordinates we can write
\begin{equation}
\label{eq:basisc5}
A =
\left(
  \begin{array}{ccccc}
    0 & 0 & -1 & 0 & 0 \\
    0 & 0 & 0 & -1 & 0 \\
    1 & 0 & 0 & 0 & 0 \\
    0 & 1 & 0 & 0 & 0 \\
    0 & 0 & 0 & 0 & 0 \\
  \end{array}
\right) \quad \textrm{and} \quad B =
\left(
  \begin{array}{ccccc}
    0 & 0 & 0 & -1 & 0 \\
    0 & 0 & 0 & 0 & -1 \\
    0 & 0 & 0 & 0 & 0 \\
    1 & 0 & 0 & 0 & 0 \\
    0 & 1 & 0 & 0 & 0 \\
  \end{array}
\right)
\end{equation}
(see \cite[Prop. 6.4]{piont-van-de-ven}). The center curve is given by $c(\lambda: \mu) = (0: 0: \mu^{2}: \mu \lambda: \lambda^{2})$.

\begin{prop}[{\cite[Prop. 6.8]{piont-van-de-ven}}]
\label{prop:orbg25}

The action of $\Aut(\GG(1, 4) \cap H^{2})$ on $X = \GG(1, 4) \cap H^{2}$ has four orbits:

\begin{itemize}
  \item $o_{1} = \{ x = [l] \in X \mid \textrm{$l$ is tangent to the center conic $C$} \}$;
  \item $o_{2} = \{ x = [l] \in X \mid \textrm{$l$ is secant to the center conic $C$} \}$;
  \item $o_{3} = \{ x = [l] \in X \mid \textrm{$l$ intersects the center conic but does not lie in the plane $P$} \}$;
  \item $o_{4} = \{ x = [l] \in X \mid \textrm{$l$ does not intersect the plane $P$} \}$.
\end{itemize}
\end{prop}

We call attention to the condition that defines the forth orbit. There are no lines in $X$ that intersects the plane $P$ but not the conic $C$. \\

\noindent \textbf{The Variety of Lines on $\GG(1, N) \cap H^{l}$.} A line contained in $\GG(k, N)$ is determined by the choice of a $k$-dimensional vector subspace $U$ and a $(k+2)$-dimensional vector subspace $V$ such that $U \subset V$. Such a line, denoted by $L_{U, V}$, is given by $L_{U, V} = \left\{ [\PP(W')] \in \GG(k, N) \mid U \subset W' \subset V \right\}$. Therefore, the variety $H_{x}$ of lines on $\GG(k, N)$ through a point $x = [\PP(W)]$ can be identified with $\PP(W) \times \PP(\CC^{N+1}/W) \cong \PP^{k} \times \PP^{N-k-1}$. Since $\Aut(\GG(k, N)) \cong \PGL(N + 1, \CC)$ for $k < (N - 1)/2$, $\GG(k, N)$ is homogeneous, and therefore $H_{x}$ does not depend on the choice of the point $x$.

For $k = 1$, we identify $H_{x}$ in the following way:
\begin{equation*}
\begin{array}{rclll}
H_{x} =
\left\{
\begin{array}{l}
\textrm{lines $\subset$ $\GG(1, N)$ passing} \\
\textrm{through $x = [\PP(W)]$}
\end{array}
\right\} &
\longleftrightarrow &
\PP(W) \times \PP(\CC^{N+1}/W) & \cong & \PP^{1} \times \PP^{N-2} \\
L_{U, V} & \longmapsto & (\PP(U), \PP(V/W)). & &
\end{array}
\end{equation*}

Let $X = \GG(1, N) \cap H^{l}$ be a linear section of codimension $l$ of $\GG(1, N)$, under the Plücker embedding. The subvariety $Z_{x} \subset H_{x}$ of lines on $X$ passing through $x \in X$ is an intersection of $l$ divisors $D_{i}$ of type $(1, 1)$ in $H_{x}$, obtained from the condition that the lines are on each hyperplane defining $H^{l}$. If $H^{l}$ is general, then by Bertini's Theorem $Z_{x}$ is a complete intersection of the $l$ general divisors $D_{i}$. If $x$ is general, then using some geometric invariant theory one can prove that $Z_{x}$ is smooth (this is true for every complex, projective and connected variety $X$; see \cite[Lemma 9]{mori-hart-conj}). For the cases treated here, $Z_{x}$ will be given explicitly. Thus, we can check the smoothness of $Z_{x}$, for $x$ general, using Jacobi's Criterion.

\section{Proof of Theorem \ref{thm:zx_g25_h2}}
\label{sec:g2.5}

In this section, we prove Theorem \ref{thm:zx_g25_h2}. To make the computations easier, we will work with the normal form (\ref{eq:basisc5}).

\begin{proof}[Proof of Theorem \ref{thm:zx_g25_h2}]
The subvariety $Z_{x}$ is the intersection of two divisors $D_{1}$ and $D_{2}$ in $H_{x} \cong \PP^{1} \times \PP^{2}$, both of type $(1, 1)$. This means that, if $[\alpha]$ and $[\beta]$ denote, respectively, the numerical classes of the pullbacks of the hyperplanes classes of $\PP^{1}$ and $\PP^{2}$ by the projections, then
\[ [Z_{x}] \equiv ([\alpha] + [\beta])^{2} \equiv 2[\alpha][\beta] + [\beta]^{2} \equiv 2[L_{1}] + [L_{2}], \]
where $[L_{1}] \equiv [\alpha][\beta]$ is the class of a line in a fiber of the first projection and $[L_{2}] \equiv [\beta]^{2}$ the class of a fiber of the second projection.

Let us write down the equations defining the subvariety $Z_{x}$. If $x = [l]$, with $l = \PP(\gen\{ p, q \})$, then any line $L_{U, V}$ on $\GG(1, 4)$ passing through $x$ is determined by two vector subspaces $U$ and $V$ given by
$$U = \gen\{ rp + sq \} \subset \gen\{ p, q \} \subset V = \gen\{ p, q, v \},$$
where $(r: s) \in \PP^{1}$ and $v \in \PP(\CC^{5}/\gen\{ p, q \})$. Such a line $L_{U, V}$ is on $X$ if and only if, every point $x' = [l'] \in L_{U, V}$, with $l' = \PP(\gen\{ rp + sq, r'p + s'q + t'v \})$, is contained in $X$, that is,
\begin{equation}
\left\{
  \begin{split}
    (rp + sq)A\pups{t}{(r'p + s'q + t'v)} & = 0 \\
    (rp + sq)B\pups{t}{(r'p + s'q + t'v)} & = 0
  \end{split}
\right. \quad \sim \quad
\left\{
  \begin{split}
    (rp + sq)A\pups{t}{v} & = 0 \\
    (rp + sq)B\pups{t}{v} & = 0.
  \end{split} \label{eq:genzg25b}
\right.
\end{equation}
The equivalence holds because since $x \in X$, we have $pA\pups{t}{q} = pB\pups{t}{q} = 0$; and since $A$ and $B$ are antisymmetric matrices, we also have $uA\pups{t}{u} = uB\pups{t}{u}$ for any $u \in \CC^{5}$. The subvariety $Z_{x}$ is defined by the equations (\ref{eq:genzg25b}). \\

Let $x = [l] \in o_{1}$, with $l$ a line tangent to the center conic $C$. As we noted in the beginning, we may work with a particular $x$. We choose $l = \PP(\gen\{ p = (0, 0, 1, 0, 0), q = (0, 0, 0, 1, 0) \})$ (recall that $C \subset P \cong \PP_{(0: 0: x_{3}: x_{4}: x_{5})}^{4}$ is given by $x_{4}^{2} - x_{3}x_{5} = 0$, so the tangent line to $C$ at $c(0: 1) = \PP(p)$ is given by $x_{5} = 0$). In this case, the system of equations (\ref{eq:genzg25b}) is
$$
\left\{
  \begin{split}
    rv_{1} + sv_{2} & = 0 \\
    sv_{1} & = 0.
  \end{split}
\right.
$$
The matrix associated to that system has only one nonzero minor, namely $M_{12} = -s^{2}$. Therefore, for $(r, s) = (1, 0)$ the subspace of solutions to that system is $4$-dimensional (containing $\gen\{ p, q \}$). Hence, there is an irreducible component of $Z_{x}$ with numerical class $[L_{1}]$ (and multiplicity 2). Clearly, the $3$-dimensional vector subspace $\gen\{ e_{3}, e_{4}, e_{5} \}$ is solution for the system for every $(r, s) \in \CC^{2}$. This means that there is an irreducible component of $Z_{x}$ with numerical class $[L_{2}]$. \\

Let $x = [l] \in o_{2}$, with $l = \PP(\gen\{ p = (0, 0, 1, 0, 0), q = (0, 0, 0, 0, 1) \})$ a secant line to the center conic $C$ through the points $c(0: 1) = \PP(p)$ and $c(1: 0) = \PP(q)$. The system of equations (\ref{eq:genzg25b}) is
$$
\left\{
  \begin{split}
    rv_{1} & = 0 \\
    sv_{2} & = 0.
  \end{split}
\right.
$$
The only nonzero minor of that system $M_{12} = rs$ says that for $(r, s) = (1, 0)$ and $(r, s) = (0, 1)$ the solution space for that system is $4$-dimensional (containing $\gen\{ p, q \}$). Therefore, there are two irreducible components of $Z_{x}$ with numerical class $[L_{1}]$. The $3$-dimensional vector subspace $\gen\{ e_{3}, e_{4}, e_{5} \}$ is the solution for the system that does not depend on $(r, s) \in \CC^{2}$. Hence, there is an irreducible component of $Z_{x}$ with numerical class $[L_{2}]$. \\

Let $x = [l] \in o_{3}$, with $l = \PP(\gen\{ p = (0, 0, 1, 0, 0), q = (q_{i}) \})$ a line intersecting the center conic $C$ at $c(0: 1) = \PP(p)$, but not contained in the plane $P$ spanned by $C$. The condition $x \in X$ implies $q_{1} = 0$, and consequently $q_{2} \neq 0$. We have the system of equations
$$
\left\{
  \begin{split}
    (r + sq_{3})v_{1} + sq_{4}v_{2} - sq_{2}v_{4} & = 0 \\
    sq_{4}v_{1} + sq_{5}v_{2} - sq_{2}v_{5} & = 0.
  \end{split}
\right.
$$
Looking at the second equation and the minor $M_{45} = s^{2}q_{2}$ we see that only for $(r, s) = (1, 0)$ the solution space for that system  is $4$-dimensional (containing $\gen\{ p, q \}$). Therefore, there is an irreducible component of $Z_{x}$ with numerical class $[L_{1}]$. Now, note that any vector $v = (v_{i}) \in \CC^{5}$ that does not depend on $(r, s)$ and is a solution for the system is contained in $\gen\{ p, q \}$, which implies that the other irreducible component of $Z_{x}$ must have numerical class $[L_{1}] + [L_{2}]$. Indeed, for $(r, s) = (1, 0)$ we get $v_{1} = 0$ and for $(r, s) = (0, 1)$ we get $v_{4} = q_{4}v_{2}/q_{2}$ and $v_{5} = q_{5}v_{2}/q_{2}$. Hence, $v = (v_{3} - v_{2}q_{3}/q_{2})p + (v_{2}/q_{2})q$. \\


Let $x \in o_{4}$. Since $x$ is general, $Z_{x}$ is smooth. By the Adjunction Formula (see \cite[Prop. II.8.20]{hartshorne}) the canonical divisor $K_{Z_{x}}$ of $Z_{x}$ is
\begin{align*}
K_{Z_{x}} & = \restr{(K_{\PP_{1} \times \PP^{2}} + D_{1} + D_{2})}{Z_{x}} \\
& = \restr{(-2[\alpha] - 3[\beta] + 2([\alpha] + [\beta])}{Z_{x}} \\
& = -[\beta] \cdot (2[\alpha][\beta] + [\beta]^{2}) \\
& = -2[P],
\end{align*}
where $[P]$ denotes the numerical class of a point. In particular, $\deg(K_{Z_{x}}) = -2$. But, as a well known consequence of Riemann-Roch Theorem, $\deg K_{Z_{x}} = 2g - 2$, where $g$ denotes the genus of $Z_{x}$. Therefore, $g = 0$, and $Z_{x} \cong \PP^{1}$.
\end{proof}

\section{Proof of Theorem \ref{thm:zx_g26_h2}}
\label{sec:g2.6}

In this section, we prove Theorem \ref{thm:zx_g26_h2}. To make computations easier, we will work with the normal form (\ref{eq:basisc6}).

\begin{proof}[Proof of Theorem \ref{thm:zx_g26_h2}]
The subvariety $Z_{x}$ is the intersection of two general divisors $D_{1}$ and $D_{2}$ in $H_{x} \cong \PP^{1} \times \PP^{3}$, both of type $(1, 1)$. Therefore, if $[\alpha]$ and $[\beta]$ denote, respectively, the numerical classes of the pullbacks of the hyperplanes classes of $\PP^{1}$ and $\PP^{3}$ by the projections, then
\[ [Z_{x}] \equiv ([\alpha] + [\beta])^{2} \equiv 2[\alpha][\beta] + [\beta]^{2} \equiv 2[P] + [L], \]
where $[P] \equiv [\alpha][\beta]$ is the class of a plane $P$ in a fiber of the first projection and $[L] \equiv [\beta]^{2}$ the class of the inverse image under the second projection of a line in $\PP^{3}$.

Let $x = [l]$, with $l = \PP(\gen\{ p, q \})$. Any line $L_{U, V}$ on $\GG(1, 5)$ passing through $x$ is determined by two vector subspaces $U$ and $V$ given by
$$U = \gen\{ rp + sq \} \subset \gen\{ p, q \} \subset V = \gen\{ p, q, v \},$$
where $(r: s) \in \PP^{1}$ and $v \in \PP(\CC^{6}/\gen\{ p, q \})$. Such a line $L_{U, V}$ is on $X$ if and only if, every point $x' = [l'] \in L_{U, V}$, with $l' = \PP(\gen\{ rp + sq, r'p + s'q + t'v \})$ is contained in $Y$, that is,
\begin{equation}
\left\{
  \begin{split}
    (rp + sq)A\pups{t}{(r'p + s'q + t'v)} & = 0 \\
    (rp + sq)B\pups{t}{(r'p + s'q + t'v)} & = 0
  \end{split}
\right. \quad \sim \quad
\left\{
  \begin{split}
    (rp + sq)A\pups{t}{v} & = 0 \\
    (rp + sq)B\pups{t}{v} & = 0.
  \end{split} \label{eq:genzg26}
\right.
\end{equation}
These are the equations defining the subvariety $Z_{x}$. \\

Let $x = [l] \in o_{1}$, with $l = \PP(\gen\{ p = (1, 0, 0, 0, 0, 0), q = (0, 0, 1, 0, 0, 0) \})$. Then, the system of equations (\ref{eq:genzg26}) becomes
$$
\left\{
  \begin{split}
    -rv_{2} - sv_{4} & = 0 \\
    -rv_{2} & = 0.
  \end{split}
\right.
$$
The only nonzero minor of that system is $M_{24} = -rs$, which says that for $(r, s) = (1, 0)$ and $(r, s) = (0, 1)$ the solution subspace for that system is $5$-dimensional (and contains $\gen\{ p, q \}$). Therefore, there are two irreducible components of $Z_{x}$ with numerical class $[P]$. Clearly the $4$-dimensional vector subspace $\gen\{ p, q, e_{5}, e_{6} \}$ is solution for the system and does not depend on $(r, s) \in \CC^{2}$. Hence, there is an irreducible component of $Z_{x}$ with numerical class $[L]$. \\

Let $x = [l] \in o_{2}$, with $l = \PP(\gen\{ p = (1, 0, 0, 0, 0, 0), q = (0, 0, 1, 0, 1, 0) \})$. Now, the system of equations (\ref{eq:genzg26}) is
$$
\left\{
  \begin{split}
    -rv_{2} - sv_{4} - sv_{6} & = 0 \\
    -rv_{2} + sv_{6} & = 0.
  \end{split}
\right.
$$
The nonzero minors of that system $M_{24} = -rs, M_{26} = -2rs$ and $M_{46} = -s^{2}$ say that $(r, s) = (1, 0)$ is the only value for which the solution subspace for that system is $5$-dimensional. Therefore, there is a unique irreducible component of $Z_{x}$ with numerical class $[P]$. Since any vector which is solution for the system and does not depend on $(r, s)$ is in the $3$-dimensional vector subspace $\gen\{ p, q, e_{5} \}$, we conclude that $Z_{x}$ does not have irreducible components with numerical class $[L]$. Therefore, the other irreducible component of $Z_{x}$ has numerical class $[P] + [L]$. \\

Let $x = [l] \in o_{3}$, with $l = \PP(\gen\{ p = (0, 0, 1, 0, 1, 0), q = (1, 0, 0, 0, 1, 0) \})$. We have the system
$$
\left\{
  \begin{split}
    -sv_{2} - rv_{4} - (r + s)v_{6} & = 0 \\
    -sv_{2} + (r + s)v_{6} & = 0
  \end{split}
\right. \quad \sim \quad
\left\{
  \begin{split}
    - rv_{4} - 2(r + s)v_{6} & = 0 \\
    -sv_{2} + (r + s)v_{6} & = 0.
  \end{split}
\right.
$$
Consider $\gen\{ p, q, e_{1}, e_{2}, e_{4}, e_{6} \}$ basis for $\CC^{6}$. Any vector $v = (v_{i}) \in \CC^{6}$ in that basis is written as $v = v_{3}p + (v_{5} - v_{3})q + (v_{1} - v_{5} + v_{3})e_{1} + v_{2}e_{2} + v_{4}e_{4} + v_{6}e_{6}$. Hence, the homogeneous coordinates of $v$ in $\PP(\CC^{6}/\gen\{ p, q \}) \cong \PP^{3}$ are $(v_{1} - v_{5} + v_{3}: v_{2}: v_{4}: v_{6}) =: (t_{i})$. With these homogeneous coordinates, $Z_{x}$ is given by the equations
\begin{equation}
\left\{
  \begin{split}
    - rt_{2} - 2(r + s)t_{3} & = 0 \\
    -st_{1} + (r + s)t_{3} & = 0.
  \end{split}
\right. \label{eq:zxxo3g26}
\end{equation}
Using Jacobi's Criterion, we can see that $Z_{x}$ is smooth. Denote by \linebreak $\pi_{2}: \PP_{(r: s)}^{1} \times \PP_{(t_{i})}^{3} \rightarrow \PP_{(t_{i})}^{3}$ the second projection. A point $(t_{i}) \in \PP^{3}$ is in the image of $Z_{x}$ under $\pi_{2}$ if and only if there is a point $(r: s) \in \PP^{1}$ such that $((r: s), (t_{i}))$ satisfies (\ref{eq:zxxo3g26}), or equivalently,
$$
\det
\left(
  \begin{array}{cc}
    -t_{2} - 2t_{3} & -2t_{3} \\
    t_{3} & t_{3} - t_{1} \\
  \end{array}
\right) = t_{1}t_{2} + 2t_{1}t_{3} - t_{2}t_{3} = 0.
$$
This is the equation of a quadric cone $Q$ with vertex $o = (1: 0: 0: 0)$. It is easy to see that the restriction of $\pi_{2}$ to $\pi_{2}^{-1}(Q \setminus \{ o \})$ is an isomorphism onto $Q \setminus \{ o \}$, and $\pi_{2}^{-1}(o) = \PP^{1} \times \{ o \} \cong \PP^{1}$. Therefore $Z_{x}$ is isomorphic to the blowup of $Q$ at the vertex $o$, or equivalently, isomorphic to the Hirzebruch surface $\FF_{2} = \PP(\Ocal_{\PP^{1}} \oplus \Ocal_{\PP^{1}}(-2))$ (see, for example, \cite[Ex. IV.18(1)]{beauville}). \\

Let $x = [l] \in o_{4}$, with $\ell = \PP(\gen\{ p = (1, 0, 1, 0, 1, 0), q = (0, 1, 0, -2, 0, 1) \})$. We have the system
$$
\left\{
  \begin{split}
    -2s(v_{3} - v_{5}) - r(v_{4} + 2v_{6}) & = 0 \\
    s(v_{1} - v_{5}) - r(v_{2} - v_{6}) & = 0.
  \end{split}
\right.
$$
Consider $\gen\{ p, q, e_{1}, e_{2}, e_{3}, e_{4} \}$ basis for $\CC^{6}$. Any vector $v = (v_{i}) \in \CC^{6}$ in that basis is written as $v = v_{5}p + v_{6}q + (v_{1} - v_{5})e_{1} + (v_{2} - v_{6})e_{2} + (v_{3} - v_{5})e_{3} + (v_{4} + 2v_{6})e_{4}$. Hence, the homogeneous coordinates of $v$ in $\PP(\CC^{6}/\gen\{ p, q \}) \cong \PP^{3}$ are $(v_{1} - v_{5}: v_{2} - v_{6}: v_{3} - v_{5}: v_{4} + 2v_{6}) =: (t_{i})$. With these homogeneous coordinates, $Z_{x}$ is given by the equations
$$
\left\{
  \begin{split}
    -2st_{2} - rt_{3} & = 0 \\
    st_{0} - rt_{1} & = 0.
  \end{split}
\right. \label{eq:zxxo4g26}
$$
Using Jacobi's Criterion, we can see that $Z_{x}$ is smooth. By Adjunction Formula and similar computations as in Section \ref{sec:g2.6} we can find the anticanonical class $-K_{Z_{x}}$ of $Z_{x}$ in $N_{1}(\PP^{1} \times \PP^{3})$; it is
\[ -K_{Z_{x}} = \restr{2 \cdot [\beta]}{Z_{x}}. \]
Now, note that there is no curve in $Z_{x}$ contracted by the second projection \linebreak $\PP_{(r: s)}^{1} \times \PP_{(t_{i})}^{3} \rightarrow \PP_{(t_{i})}^{3}$. This implies that $-K_{Z_{x}}$ is ample, with index $i_{Z_{x}} = 2$. By Kobayashi-Ochiai's Theorem (see \cite[Cor. 3.1.15]{ag-iv-fanovarieties}), $Z_{x}$ is isomorphic to a smooth quadric in $\PP^{3}$.
\end{proof}

\section{Weakly $2$-Fano Manifolds}
\label{sec:w2fano}

As an application of the results from the previous sections, we complete the classification of weakly $2$-Fano manifolds, initiated in \cite{araujo-castravet}. A smooth, complex, projective variety $X$ with second Chern character $\ch_{2}(X)$ is \textit{weakly $2$-Fano} if
$$\ch_{2}(X) \cdot [S] \geq 0,$$
for all surface $S \subset X$. Weakly $2$-Fano manifolds were introduced by de Jong and Starr in \cite{dejong-starr} and further studied by Araujo and Castravet in \cite{araujo-castravet-pol} and \cite{araujo-castravet}. This notion is related to the problem of finding sections of fibrations over surfaces, and with the notion of rational simple connectedness introduced by de Jong and Starr. In \cite{araujo-castravet}, Araujo and Castravet gave an almost complete classification of weakly $2$-Fano manifolds of dimension $n \geq 3$ and index at least $n - 2$. The only cases left open were the general linear sections $\GG(1, 4) \cap H^{2}$ and $\GG(1, 5) \cap H^{2}$ are weakly $2$-Fano or not. Here we will prove that these manifolds are not weakly \linebreak $2$-Fano. Let $X = \GG(1, N) \cap H^{2}$ be a general linear section of codimension $2$ of the Grassmannian $\GG(1, N)$, $N \geq 3$. By \cite[Prop. 31]{araujo-castravet} we have
\begin{equation}
\ch_{2}(X) = \left( \frac{N-3}{2} \right) \restr{\sigma_{2}}{X} - \left( \frac{N-3}{2} \right) \restr{\sigma_{1, 1}}{X}, \label{eq:ch2_g2n_h2}
\end{equation}
where $\sigma_{2}$ and $\sigma_{1, 1}$ are Schubert cycles, generators of the $(2N - 4)$-th graded piece of the Chow ring of $\GG(1, N)$. The strategy is to find a surface $S \subset X$ with class $\sigma_{1, 1}^{\ast}$ in $\GG(1, N)$, for $N = 4, 5$. By Duality Theorem we will have
\[ \ch_{2}(X) \cdot [S] = \ch_{2}(X) \cdot \sigma_{1, 1}^{\ast} = - \left( \frac{N-3}{2} \right) < 0, \quad \textrm{for} \ N = 4, 5. \]

\begin{cor}
The general linear sections of codimension $2$ of Grassmannians $\GG(1, 4) \cap H^{2}$ and $\GG(1, 5) \cap H^{2}$ under the Plücker embedding are not weakly $2$-Fano.
\end{cor}

\begin{proof}
Let $X = \GG(1, 4) \cap H^{2}$. Take $x = [l] \in X$, with $l = \PP(W)$, in the first or second orbit of the action of $\Aut(X)$ on $X$ (see Proposition \ref{prop:orbg25}). The variety $Z_{x}$ of lines passing through $x$ and contained in $X$ has an irreducible component $C$ with numerical class $[L_{2}]$ in $\PP(W) \times \PP(\CC^{5}/W)$ by Theorem \ref{thm:zx_g25_h2}. Consider the universal family morphisms
\begin{equation}
\xymatrix{
\Ucal_{x} \ar[d]^{\pi_{x}} \ar[r]^{e_{x}} & X \\
Z_{x}
}. \label{diag:fam_mor_lines_x}
\end{equation}
The surface $S$ defined by $S = e_{x}(\pi_{x}^{-1}(C)) \subset X$ has Schubert class $\sigma_{1, 1}^{\ast} = \sigma_{2, 2}$ in $\GG(1, 4)$, and $\ch_{2}(X) \cdot [S] = - \frac{1}{2} < 0$. Therefore, $X = \GG(1, 4) \cap H^{2}$ is not weakly $2$-Fano. \\

Now let $Y = \GG(1, 5) \cap H^{2}$. Take $x = [l] \in Y$, with $l = \PP(W)$, in the first orbit of the action of $\Aut(Y)$ on $Y$ (see Proposition \ref{prop:orbg26}). The variety $Z_{x}$ of lines passing through $x$ and contained in $Y$ has an irreducible component $C$ with numerical class $[L]$ in $\PP(W) \times \PP(\CC^{6}/W)$ by Theorem \ref{thm:zx_g26_h2}. Consider the universal family morphisms analogue to (\ref{diag:fam_mor_lines_x}). The surface $S$ defined by $S = e_{x}(\pi_{x}^{-1}(C)) \subset Y$ has Schubert class $\sigma_{1, 1}^{\ast} = \sigma_{3, 3}$ in $\GG(1, 5)$, and $\ch_{2}(X) \cdot [S] = - 1 < 0$.
Therefore, $Y = \GG(1, 5) \cap H^{2}$ is not weakly $2$-Fano.
\end{proof}


\bibliographystyle{alpha}

\bibliography{biblio}


\end{document}